\documentclass[a4paper]{article}

\usepackage{tikz}

\usepackage{amssymb}
\usepackage[leqno]{amsmath}
\usepackage{amsfonts}
\usepackage{amsopn}
\usepackage{amstext}
\usepackage{amsthm}

\def\subjclass#1{{\renewcommand{\thefootnote}{}%
\footnote{\emph{Mathematics Subject Classification (2020):} #1}}}

\def\keywords#1{{\renewcommand{\thefootnote}{}%
\footnote{\emph{Keywords:} #1}}}

\def\ackn#1{{\renewcommand{\thefootnote}{}%
\footnote{#1}}}

\newtheorem{theorem}{Theorem}[section]
\newtheorem{lemma}[theorem]{Lemma}
\newtheorem{proposition}[theorem]{Proposition}
\newtheorem{fact}[theorem]{Fact}
\newtheorem{claim}[theorem]{Claim}

\newcommand\norm[1]{\left\|#1\right\|}
\newcommand\zero{\mathbf{0}}
\newcommand{\fra}{Fra\"iss\'e}
\newcommand\Aut{\mbox{Aut}}
\newcommand\Iso{\mbox{Iso}}
\newcommand\LIso{\mbox{LIso}}

\DeclareMathOperator{\spn}{span}

\newcommand\FFb{\mathbb{F}}
\newcommand\NNb{\mathbb{N}}
\newcommand\QQb{\mathbb{Q}}
\newcommand\RRb{\mathbb{R}}

\newcommand\KKc{\mathcal{K}}
\newcommand\LLc{\mathcal{L}}
\newcommand\PPc{\mathcal{P}}
\newcommand\UUc{\mathcal{U}}

\newcommand\GGf{\mathbf{G}}
\newcommand\HHf{\mathbf{H}}
\newcommand\LLf{\mathbf{L}}
\newcommand\UUf{\mathbf{U}}

\makeatletter
\newcommand{\opnorm}{\@ifstar\@opnorms\@opnorm}
\newcommand{\@opnorms}[1]{%
	\left|\mkern-1.5mu\left|\mkern-1.5mu\left|
	#1
	\right|\mkern-1.5mu\right|\mkern-1.5mu\right|
}
\newcommand{\@opnorm}[2][]{%
	\mathopen{#1|\mkern-1.5mu#1|\mkern-1.5mu#1|}
	#2
	\mathclose{#1|\mkern-1.5mu#1|\mkern-1.5mu#1|}
}
\makeatother

\author{Ond\v{r}ej Kurka \and Maciej Malicki}
\date{}

\begin{document}
\title{The rational Gurarii space and its linear isometry group}

\maketitle

\begin{abstract}
We show that the classes of partial isometries in finite-dimensional polyhedral spaces and in finite-dimensional rational polyhedral spaces do not have the weak amalgamation property. This implies that the linear isometry group of the rational Gurarii space does not have a comeager conjugacy class. Our methods demonstrate also that the classes of finite-dimensional polyhedral spaces and of finite-dimensional rational polyhedral spaces fail to have the Hrushovski property.
\end{abstract}



\ackn{The first named author was supported by the Czech Science Foundation, project no.~GA\v{C}R~22-07833K, and by the Academy of Sciences of the Czech Republic (RVO 67985840). The second author was partially supported by the National Science Centre, Poland under the Weave-UNISONO call in the Weave programme [grant no 2021/03/Y/ST1/00072].}

\section{Introduction}
In this paper, we investigate amalgamation of partial isometries in finite-dimensional polyhedral spaces. We prove that the classes $\PPc_\QQb$ and $\PPc_\RRb$ of finite-dimensional polyhedral spaces over $\QQb$ and over $\RRb$, respectively, do not satisfy the Hrushovski property, and that the classes $(\PPc_\QQb)_1$ and $(\PPc_\RRb)_1$ of partial isometries in the respective spaces do not have the weak amalgamation property. In particular, this implies that the linear isometry group $\Aut(\GGf_\QQb)$ of the \fra \ limit $\GGf_\QQb$ of $\PPc_\QQb$, called the rational Gurarii space, does not admit a comeager conjugacy class.    

Our investigation is motivated by a question posed by Sabok \cite{Sa}, namely whether the linear isometry group $\LIso(\GGf)$ of the Gurarii space $\GGf$ has the automatic continuity property. Recall that a Polish (i.e., separable and completely metrizable) group has the automatic continuity property if all its homomorphisms into separable topological groups are continuous. This notion was first studied in the setting of $C^*$-algebras and Banach algebras. For example, Sakai \cite{Sak}, resolving a conjecture of Kaplansky, showed that all derivations on a $C^*$-algebra are norm-continuous. More recently, Tsankov \cite{Ts} established the automatic continuity for the unitary group $U(\HHf)$ of the separable Hilbert space $\HHf$, while Sabok \cite{Sa} obtained the same for the isometry group $\Iso(\UUf)$ of the Urysohn space $\UUf$.

He developed a general framework (see also \cite{Ma1} for a similar but simpler approach), where properties of automorphism groups of complete metric structures are studied via certain countable substructures, and their automorphism groups. His proof ultimately rests on a result of Solecki \cite{So}, who showed that the class $\UUc_0$ of finite metric spaces with rational distances has the Hrushovski property. Consequently, the isometry group $\Aut(\UUf_0)$ of the rational Urysohn space $\UUf_0$, the \fra \ limit of $\UUc_0$, has the automatic continuity property. 

As a matter of fact, an intermediate property, called ample generics, plays a key role there. Recall that a Polish group $G$ has ample generics if it has a comeager diagonal $n$-conjugacy class for every $n \in \NNb$, i.e., there exists a comeager orbit under the action of $G$ on $G^n$ by diagonal conjugation. It was proved by Kechris and Rosendal \cite{KeRo} that a group with ample generics has the automatic continuity property. Let us mention that Tsankov's work on $U(\HHf)$ also requires that certain countable counterpart of the separable Hilbert space, studied by Rosendal \cite{Ro}, has ample generics. 

It is known that $\LIso(\GGf)$ does not have ample generics. In fact, it does not even admit a comeager conjugacy class -- indeed, \cite{DoMa} shows that no countably infinite group has a generic representation in $\LIso(\GGf)$. Nevertheless, the question of whether this group has the automatic continuity property remains open. Turning to the rational Gurarii space $\GGf_\QQb$, essentially nothing has been established in this direction so far. In this paper we prove that $\Aut(\GGf_\QQb)$ has no comeager conjugacy class, which indicates that Sabok's method may not extend to the Gurarii space. Still, it remains conceivable that a different notion of a rational Gurarii space could be better suited for such an approach.

\section{\fra \ classes}

In model theory, a structure $M$ is a set with relations, functions, and constants. The symbols denoting these objects, together with their arities, are referred to as the signature of $M$. 

Let $\KKc$ be a class of structures in a fixed signature that is closed under isomorphism. Following \cite{KeRo}, we say that $\KKc$ is \emph{hereditary} if it is closed under taking substructures; it has \emph{joint embedding property} (JEP) if any two $A,B \in \KKc$ can be embedded into a single $C \in \KKc$; it has \emph{amalgamation property} (AP) if for every $A, B,C \in \KKc$ and embeddings $\alpha\colon A\to B$ and $\beta\colon A\to C$ there is $D\in\KKc$ and embeddings $\gamma\colon B\to D$, $\delta\colon C\to D$ such that $\gamma\circ\alpha=\delta\circ\beta$. 
Finally, $\KKc$ has \emph{weak amalgamation property} (WAP) if for every $A\in\KKc$ there is $A'\in\KKc$ and an embedding $\phi\colon A\to A'$ such that for every $B,C\in\KKc$ and embeddings $\alpha\colon A'\to B$, $\beta\colon A'\to C$ there is $D\in\KKc$ and embeddings $\gamma\colon B\to D$, $\delta\colon C\to D$ such that $\gamma\circ\alpha\circ\phi=\delta\circ\beta\circ\phi$. Note that since $\KKc$ is closed under isomorphism, we can assume that in the definitions above, only extensions $\phi$, $\alpha$, $\beta$, instead of all embeddings, are considered.  

Let $\KKc$ be a class of finitely generated structures in a fixed signature that is closed under isomorphism, and countable up to isomorphism. We say that $\KKc$ is a \emph{\fra \ class} if it is hereditary, has JEP and AP. By the classical theorem due to \fra \ (see \cite{Fr}), if $\KKc$ is a \fra\ class, then there exists a unique up to isomorphism countable structure $M$ such that  $\KKc$ is the class of finitely generated structures embeddable in $M$, and $M$ has the extension property, i.e., for any $A,B \in \KKc$ with $A \subseteq B$, and embedding $\alpha: A \to M$, there is an embedding $\beta: B \to M$ that extends $\alpha$. We call this $M$ the \emph{\fra \ limit} of $\KKc$.

A class $\KKc$ of structures has the \emph{$n$-Hrushovski property}, if for any $A \in \KKc$, and $n$-tuple $f_1, \ldots, f_n$ of partial automorphisms of $A$, there exists $B \in \KKc$ such that $A \subseteq B$, and every $f_i$ can be extended to an automorphism of $B$. It has the \emph{Hrushovski property} if it has the $n$-Hrushovski property for every $n \in \NNb$.

For a class $\KKc$ of structures, let $\KKc_1$ be the class of all partial automorphisms of elements from $\KKc$, i.e., pairs $(A,f:B \to C)$ such that $A,B,C \in \KKc$, $B,C \subseteq A$, and $f$ is an isomorphism between $B$ and $C$. Then $\alpha:A \to A'$ is an embedding of $(A,f:B \to C)$ into $(A',f':B' \to C')$ if it embeds $A$ into $A'$ and $\alpha \circ f \subseteq f' \circ \alpha$. In the same way, we define classes $\KKc_n$ of elements of the form $(A,f_1:B_1 \to C_1, \ldots, f_n:B_n \to C_n )$, where $A \in \KKc$ and $f_i$ are partial automorphisms of $A$. JEP, AP and WAP can be defined for $\KKc_n$ analogously to $\KKc$.

Recall that a topological group is \emph{Polish} if it is separable, and completely metrizable. A Polish group $G$ has a dense (non-meager, etc.) diagonal $n$-conjugacy class if there is a dense (non-meager, etc.) orbit in the action of $G$ on $G^n$ defined by $g.(g_1, \ldots, g_n)=(gg_1g^{-1}, \ldots, gg_ng^{-1})$, $g,g_1, \ldots, g_n \in G$. It has \emph{ample generics} if it has a comeager diagonal $n$-conjugacy class for every $n \in \NNb$.

It is well known that the automorphism group $\Aut(M)$ of a countable structure $M$, with the topology inherited from $M^M$, where $M$ is regarded as a discrete space, is a Polish group. The following characterization was proved for \fra \ limits of classes $\KKc$ \ of finite structures in \cite{KeRo}, and generalized to classes $\KKc$ of finitely generated structures (and any countable structures $M$) in \cite{Ma2}: 

\begin{theorem}
\label{th:kechros}
Let $M$ be the \fra \ limit of a \fra \ class $\KKc$, and let $n \geq 1$. The group $\Aut(M)$ has a comeager diagonal $n$-conjugacy class iff $\KKc_n$ has JEP and WAP.
\end{theorem}

\section{Polyhedral spaces}

We now recall the notion of a polyhedral space, considered here both in the rational and real settings. First, let us formulate the following fact that is well known (at least in the real setting).

In what follows, in a vector space over $ \mathbb{Q} $, the notions of a convex hull, an absolutely convex hull and an extreme point are defined in the same way as in a real vector space, with the difference that rational numbers are used instead of real numbers.

\begin{fact} \label{fact1}
Let $ X $ be a finite-dimensional normed space over $ \mathbb{F} = \mathbb{Q} $ or $ \mathbb{F} = \mathbb{R} $. Then the following assertions are equivalent:

(i) The closed unit ball $ B_{X} $ of $ X $ is the absolutely convex hull of a finite number of points in $ X $.

(ii) There are functionals $ u^{*}_{1}, \dots, u^{*}_{n} $ in the dual space $ X^{*} = \mathcal{L}(X, \mathbb{F}) $ such that the norm of $ X $ can be expressed by
$$ \Vert x \Vert = \max \{ |u^{*}_{i}(x)| : 1 \leq i \leq n \}. \leqno (*) $$
\end{fact}

\begin{proof}
We consider this fact in the setting $ \mathbb{F} = \mathbb{R} $ to be well known and understood, so we concentrate only on the setting $ \mathbb{F} = \mathbb{Q} $. We can assume that $ X = \mathbb{Q}^{d} $ for some $ d $. By $ e_{1}, \dots, e_{d} $ we denote the canonical basis of $ \mathbb{Q}^{d} $.

(i) $ \Rightarrow $ (ii): Let the unit ball of $ X = \mathbb{Q}^{d} $ be the absolutely convex hull of points $ a_{1}, a_{2}, \dots, a_{m} $ in $ \mathbb{Q}^{d} $. Let $ B $ denote the absolutely convex hull in $ \mathbb{R}^{d} $ of the same set of points. Since $ B $ contains $ \varepsilon e_{i} $ for a sufficiently small $ \varepsilon > 0 $ (we can take $ 1/\Vert e_{i} \Vert_{X} $) for $ 1 \leq i \leq d $, the point $ 0 \in \mathbb{R}^{d} $ belongs to the interior of $ B $. Hence $ B $ is the unit ball of a norm $ \opnorm{\cdot} $ on $ \mathbb{R}^{d} $. Let us realize that $ \opnorm{x} = \Vert x \Vert_{X} $ for every $ x \in \mathbb{Q}^{d} $, or equivalently $ B_{X} = B \cap \mathbb{Q}^{d} $. We need to check that if a point $ x \in \mathbb{Q}^{d} $ can be expressed as a convex combination of $ \pm a_{1}, \dots, \pm a_{m} $, then it can be expressed as a convex combination of $ \pm a_{1}, \dots, \pm a_{m} $ with rational coefficients. Note that $ x $ is a convex combination of at most $ d+1 $ points $ \pm a_{k} $ that form vertices of a simplex. Then the coefficients are uniquely determined, and since $ x $ and all $ a_{k} $'s belong to $ \mathbb{Q}^{d} $, the coefficients must be rational.

We know from the real case that there are functionals $ u^{*}_{1}, \dots, u^{*}_{n} $ such that the norm of $ X $ can be expressed by $ (*) $. However, we need to show that these functionals can be chosen to be $ \mathbb{Q} $-valued on $ \mathbb{Q}^{d} $, or equivalently $ u^{*}_{j}(e_{i}) \in \mathbb{Q} $ for $ 1 \leq j \leq n, 1 \leq i \leq d $. For this purpose, let $ u^{*}_{1}, \dots, u^{*}_{n} $ be the extreme points of the dual unit ball of the space $ (\mathbb{R}^{d}, \opnorm{\cdot}) $. Equivalently, $ u^{*}_{1}, \dots, u^{*}_{n} $ are those functionals that are equal to $ 1 $ on some $ (d-1) $-dimensional polytope contained in the unit sphere. Thus, each $ u^{*}_{j} $ is equal to $ 1 $ for $ d $ linearly independent points $ \pm a_{k} $. These $ d $ points have rational coordinates and uniquely determine the numbers $ u^{*}_{j}(e_{i}) $, so these numbers must be rational.

(ii) $ \Rightarrow $ (i): Let $ u^{*}_{1}, \dots, u^{*}_{n} $ in the dual space $ X^{*} = \mathcal{L}(\mathbb{Q}^{d}, \mathbb{Q}) $ be such that the norm of $ X $ can be expressed by $ (*) $. Let us extend each $ u^{*}_{j} $ to $ \mathbb{R}^{d} $ in the obvious way and equip $ \mathbb{R}^{d} $ with $ \opnorm{\cdot} $ given by the same formula $ (*) $. This is clearly a pseudonorm, and we check first that it is a norm. Given $ x = (x_{1}, \dots, x_{d}) \in \mathbb{R}^{d} $ different from $ 0 $, we need to verify that $ \opnorm{x} > 0 $. Assume the opposite, i.e., $ \opnorm{x} = 0 $, and so $ u^{*}_{j}(x) = 0 $ for $ 1 \leq j \leq n $. Let $ U $ denote the matrix
$$ U = (u^{*}_{j}(e_{i}))_{1 \leq j \leq n, 1 \leq i \leq d}. $$
Considering $ x $ as a one column matrix, we obtain $ Ux = 0 $, and so $ U $ has rank less than $ d $ when considered as a matrix over $ \mathbb{R} $. At the same time, $ U $ is a matrix consisting of rational numbers, and it has the same rank if considered over $ \mathbb{Q} $. This is because the rank of $ U $ is equal to the size of the biggest square submatrix with non-zero determinant. Hence, there is $ \widetilde{x} \in \mathbb{Q}^{d} $ different from $ 0 $ such that $ U\widetilde{x} = 0 $. But then $ \Vert \widetilde{x} \Vert_{X} = 0 $, which is not possible. Thus $ \opnorm{\cdot} $ is a norm indeed.

Let $ a_{1}, \dots, a_{m} $ be the extreme points of the unit ball of $ (\mathbb{R}^{d}, \opnorm{\cdot}) $ (we know from the real case that there are only finitely many of them). Each $ a_{k} $ belongs to $ \mathbb{Q}^{d} $, since $ u^{*}_{j}(a_{k}) = \pm 1 $ for $ d $ linearly independent functionals $ u^{*}_{j} $, while these functionals are $ \mathbb{Q} $-valued on $ \mathbb{Q}^{d} $ and together with the values $ \pm 1 $ uniquely determine $ a_{k} $.

To show that the points $ a_{1}, \dots, a_{m} $ witness that (i) holds, we need to check that if a point $ x \in \mathbb{Q}^{d} $ can be expressed as a convex combination of $ a_{1}, \dots, a_{m} $, then it can be expressed as a convex combination of $ a_{1}, \dots, a_{m} $ with rational coefficients. This can be done in the same way as in the proof of the implication (i) $ \Rightarrow $ (ii).
\end{proof}

The spaces that satisfy the equivalent conditions (i) and (ii) above are called \emph{polyhedral}. For a polyhedral space $ X $, the extreme points of the unit ball $ B_{X} $ are called \emph{vertices} of $ B_{X} $. In other words, the set of vertices of $ B_{X} $ is the smallest set whose convex hull is $ B_{X} $.

Let us note that if the dual space $ X^{*} $ is polyhedral and $ u^{*}_{1}, \dots, u^{*}_{n} $ is the enumeration of vertices of $ B_{X^{*}} $, then $ (*) $ holds for each $ x \in X $.

\begin{fact}
(a) If $ X $ is a polyhedral space, then its dual space $ X^{*} = \mathcal{L}(X, \mathbb{F}) $ is polyhedral.

(b) If $ Y $ is a subspace of a polyhedral space $ X $, then both $ Y $ and the quotient $ X/Y $ are polyhedral.

(c) If $ X $ and $ Y $ are polyhedral spaces, then their $ \ell_{\infty} $-sum and $ \ell_{1} $-sum, that is, the space $ X \oplus Y $ with the norms $ \Vert (x, y) \Vert_{\infty} = \max \{ \Vert x \Vert, \Vert y \Vert \} $ and $ \Vert (x, y) \Vert_{1} = \Vert x \Vert + \Vert y \Vert $, are polyhedral spaces.

(d) If $ f : X \to X $ is a surjective linear isometry on a polyhedral space $ X $, then there is $ n \in \NNb $ such that $ f^{n} $ is the identity on $ X $.
\end{fact}

\begin{proof}
(a) We know from the above comment that if $ X^{*} $ is polyhedral, then $ X $ is polyhedral. Thus, if $ X^{**} $ is polyhedral, then $ X^{*} $ is polyhedral. It remains to note that $ X^{**} $ is isometric to $ X $.

(b) It is easy to see that $ Y $ is polyhedral. The dual of $ X/Y $ is isometric to the annihilator $ Y^{\bot} = \{ u^{*} \in X^{*} : (\forall y \in Y) (u^{*}(y) = 0) \} $ which, being a subspace of $ X^{*} $, is polyhedral. It follows that $ X/Y $ is polyhedral.

(c) If the norm on $ X $ is expressed by $ \Vert x \Vert = \max \{ |u^{*}_{i}(x)| : 1 \leq i \leq n \} $ and the norm on $ Y $ is expressed by $ \Vert y \Vert = \max \{ |v^{*}_{j}(y)| : 1 \leq j \leq m \} $, then
$$ \Vert (x, y) \Vert_{\infty} = \max \big( \{ |u^{*}_{i}(x)| : 1 \leq i \leq n \} \cup \{ |v^{*}_{j}(y)| : 1 \leq j \leq m \} \big), \quad $$
so the $ \ell_{\infty} $-sum of $ X $ and $ Y $ is polyhedral. Concerning the $ \ell_{1} $-sum of $ X $ and $ Y $, we just point out that its dual is isometric to the $ \ell_{\infty} $-sum of $ X^{*} $ and $ Y^{*} $.

(d) Considering the restriction of $ f $ to the vertices of $ B_{X} $, we obtain a permutation $ \pi $. It is sufficient to note that $ \pi^{n} $ is the identity permutation for some $ n $.
\end{proof}

\section{Two classes of linear spaces}

Let $\FFb$ be a field, and let $L_\FFb$ be the signature consisting of functional symbols $f_{\alpha_1, \ldots, \alpha_n}(x_1, \ldots, x_n)$, $\alpha_i \in \FFb$, and a constant symbol $\zero$. Let $N_\FFb$ be the signature obtained by expanding $L_\FFb$ by unary relational symbols $\norm{\cdot}_\alpha$, $\alpha \in \FFb$. 
Every linear space $V$ over $\FFb$ can be regarded as a structure in the signature $L_\FFb$ by defining functions $f_{\alpha_1, \ldots, \alpha_n}(x_1, \ldots, x_n)=\alpha_1 x_1+ \ldots +\alpha_n x_n$, and $\zero$ as the zero vector. Analogously, a linear space over $\FFb$ with a norm $\norm{\mbox{ }}$ taking values in $\FFb$  can be regarded as a structure in the signature $N_\FFb$ by putting $\norm{x}_\alpha$ iff $\norm{x}=\alpha$.

Let $\LLc_\QQb$ ($\LLc_\RRb$) be the class of finitely-dimensional linear spaces over $\QQb$ ($\RRb$) (and linear injections as embeddings). Let $\PPc_\QQb$ ($\PPc_\RRb$) be the class of finitely-dimensional linear spaces over $\QQb$ ($\RRb$), with polyhedral norms taking values in $\QQb$ ($\RRb$) (and linear isometric injections as embeddings). The following is well known and easy to verify:

\begin{proposition}
$\LLc_\QQb$, $\LLc_\RRb$, $\PPc_\QQb$, and $\PPc_\RRb$ have JEP and WAP. In particular, because $\LLc_\QQb$ and $\PPc_\QQb$ are countable up to isomorphism, these classes are \fra. 
\end{proposition}

The \fra \ limit $\LLf_\QQb$ of $\LLc_\QQb$ is the infinite-dimensional countable linear space over $\QQb$.
The \fra \ limit $\GGf_\QQb$ of $\PPc_\QQb$ is a normed linear space over $\QQb$ that we will call the \emph{rational Gurarii space}. Clearly, $\Aut(\LLf_\QQb)$ is the linear isomorphism group of $\LLf_\QQb$, and $\Aut(\GGf_\QQb)$ is the linear isometry group of $\GGf_\QQb$.

Recall that the \emph{Gurarii space} $\GGf$ is the unique separable Banach space having the property that for any $\epsilon> 0$, finite dimensional Banach spaces $E \subseteq F$, and isometric embedding $\phi : E \to \GGf$, there is a linear embedding $\Phi : F \to \GGf$ extending $\phi$ such that in addition $(1-\epsilon) \norm{x} < \norm{\Phi(x)}< (1+\epsilon) \norm{x}$, for all $x \in F \setminus \{ 0 \} $. It can also be characterized (see, e.g., \cite[Theorem 2.7]{GaKu}), as the unique separable Banach space such that for any $\epsilon> 0$, finite-dimensional Banach spaces $E \subseteq F$, and isometric embedding $\phi : E \to \GGf$, there is an isometric embedding $\Phi : F \to \GGf$ such that $\norm{\Phi \upharpoonright E - \phi}<\epsilon$. Note that finite-dimensional rational Banach spaces, as defined in \cite[Section 2.2]{GaKu2}, are exactly completions of spaces from $\PPc_\QQb$. Moreover, $\GGf$ is the completion of $\GGf_\QQb$. We sketch the proof of this folklore result for the sake of comprehensiveness.

\begin{proposition}
The completion of the rational Gurarii space $\GGf_\QQb$ is the Gurarii space $\GGf$. 
\end{proposition}

\begin{proof}
Let $\HHf$ be the completion of $\GGf_\QQb$, and let $\norm{ \ }$ denote the norm on $\GGf$. We start with the following observation. 

\emph{Claim}: Fix $\epsilon>0$, set $A=\{a_0, \ldots, a_k\}$ of unit vectors in $\GGf_\QQb$, $f:A \to \GGf$ that gives rise to an isometry between $\spn A$ (in $\HHf$) and $\spn f[A]$, and a unit vector $y \in \GGf$. There are unit vectors $y'_0, \ldots, y'_k ,y' \in \GGf$ such that
\begin{itemize}
\item $\norm{f(a_i)-y'_i}<\epsilon$, $\norm{y-y'}<\epsilon$,
\item $a_i \mapsto y'_i$ gives rise to an isometric isomorphism of $\spn A$ and $\spn \{y'_0, \ldots, y'_k\}$,
\item the linear space over $\QQb$ spanned by $\{y'_0, \ldots, y'_k,y'\}$ is in $\PPc_\QQb$. 
\end{itemize}

To show the claim, note that Lemma 2.4 in \cite{GaKu2} and its proof gives that there exists a norm $\norm{ \ }'_0$ on $B'_0=\spn \{f(a_0), \ldots,f(a_k),y \}$ that is $\epsilon$-equivalent to $\norm{ \ }$ on $B'_0$ (i.e., $(1-\epsilon) \norm{x} < \norm{x}'_0< (1+\epsilon) \norm{x}$ for $x \in B'_0, x \neq \zero $), equal to $\norm{ \ }$ on $\spn \{f(a_0), \ldots, f(a_k)\}$, equal to $1$ at $y$, and that turns $B'_0$ into a rational space in the sense that the linear space over $\QQb$ spanned by $\{f(a_0), \ldots, f(a_k),y\}$ equipped with $\norm{ \ }'_0$ is in $\PPc_\QQb$. It is easy to observe (see, e.g., \cite[Lemma 1.6]{BYHe}) that then there exists a semi-norm $\norm{ \ }'$ on $B'=(B'_0,\norm{ \ }) \oplus (B'_0,\norm{ \ }'_0)$ that extends $\norm{ \ }$ and $\norm{ \ }'_0$, and, moreover, $\norm{(f(a_i),\zero)-(\zero,f(a_i))}'<\epsilon$, $\norm{(y,\zero)-(\zero,y)}'<\epsilon$. By the above characterization of the Gurarii space, we can assume that actually $B'/\norm{ \ }'$ is a subspace of $\GGf$ such that $\norm{f(a_i)-(\zero,f(a_i))}<\epsilon$, $\norm{y-(\zero,y)}<\epsilon$. 
In other words, $y'_i=(\zero,f(a_i))$, $y'=(\zero,y)$ witness that the claim holds.

Now, to finish the proof of the proposition, we will construct an isometric injection $f:\GGf_\QQb \to \GGf$ whose image is dense in $\GGf$, and $f(\zero)=\zero$. Let $X=\{x_i\}$ be all the unit vectors in $\GGf_\QQb$, and let $\{y_i\}$ be a dense subset of $\GGf \setminus \{0\}$. Put $f_0(\zero)=\zero$, and suppose that injections $f_i: X_i \to \GGf$, $i \leq n$, have been defined, so that, for $1  \leq i \leq n$,
\begin{itemize}
\item $X_i \subseteq X$ is finite, and $X_{i-1} \subseteq X_{i}$,
\item $f_i$ gives rise to an isometry between $\spn X_i$ and $\spn f_i[X_i]$,
\item $\norm{f_{i}(x)-f_{i-1}(x)}<2^{-i}$, for $x \in X_{i-1}$,
\item $x_i \in \spn X_{2i}$, if $2i \leq n$,
\item $\norm{y-y_i}<2^{-i}$ for some $y  \in \spn f_{2i+1}[X_{2i+1}]$, if $2i+1 \leq n$.
\end{itemize}
For odd $n$, we use the above characterization of the Gurarii space to construct a required $f_{n+1}$. For even $n$, suppose that $y_{n/2} \not \in \spn f_n[X_n]$. We fix $y'_0, \ldots, y'_k$, $y' \in \GGf$ as in the claim for a sufficiently small $\epsilon$, $a_0, \ldots, a_k$ enumerating $X_n$, $f=f_n$ and $y=y_{n/2}/\norm{y_{n/2}}$. Then, applying amalgamation in $\PPc_\QQb$, we find $l \in \NNb$ such that $f_{n+1}$ defined by $f_{n+1}(a_i)=y'_i$, $f_{n+1}(x_l)=y'$ witnesses that $\spn X_n \cup \{x_l\}$ (in $\HHf$) is isometric with $\spn f_{n+1}[X_n] \cup \{y'\}$. Finally, we define $f(x_n)=\lim_m f_m(x_n)$.
\end{proof}

Now we point out that $\Aut(\LLf_\QQb)$, similarly to $\Aut(\HHf_\QQb)$, where $\HHf_\QQb$ is a countable counterpart of the separable Hilbert space studied in \cite{Ro}, has ample generics.

\begin{proposition}
The classes $\LLc_\QQb$ and $\LLc_\RRb$ have the Hrushovski property. In particular, $\Aut(\LLf_\QQb)$ has ample generics.      
\end{proposition}

\begin{proof}
We show that $\LLc_\QQb$ has the $n$-Hrushovski property, and that $(\LLc_\QQb)_n$ has WAP, for $n=1$. The proofs for $n>1$, and for $\LLc_\RRb$, are analogous. Fix $(X,f:E \to F) \in (\LLc_\QQb)_1$. Without loss of generality, we can assume that $E \cup F$ generates $X$. Fix a basis $x_0, \ldots, x_n$ in $X$ such that, for some $j,k \leq n$, $E=\spn \{x_0, \ldots, x_k\}$, and $F=\spn \{x_j, \ldots, x_n\}$ (i.e., $E \cap F= \spn \{x_j, \ldots, x_k\}$). Since the range of $f$ is $F$, we have $n=j+k$. We extend $f$ by putting $x_i=f(x_{n-i})$ for $i<j$.

It follows that there is a subfamily $\mathcal{F} \subseteq (\LLc_\QQb)_1$ that has AP, and is cofinal in $(\LLc_\QQb)_1$, i.e., for every $(X,f:E \to F) \in (\LLc_\QQb)_1$ there is $(X',f':E' \to F') \in \mathcal{F}$ that embeds $(X,f:E \to F)$. In particular, $(\LLc_\QQb)_1$ has weak amalgamation. Indeed, for $(A,f:B \to C) \in (\LLc_\QQb)_1$, fix $(A_0,f_0:A_0 \to A_0) \in (\LLc_\QQb)_1$ such that $f_0$ extends $f$. As $A_0$ is obviously invariant under every extension of $f_0$, for any $(A_1,f_1:B_1 \to C_1), (A_2,f_2:B_2 \to C_2) \in (\LLc_\QQb)_1$ such that $f_1$, $f_2$ extend $f_0$, we can assume that $A_1 \cap A_2=A_0$. Write $A_1=A_0 \oplus A'_1$, $f_1=(f_0,f'_1)$, $A_2=A_0 \oplus A'_2$, $f_2=(f_0,f'_2)$. It is immediate that $(A_0 \oplus A'_1 \oplus A'_2,(f_0,f'_1, f'_2))$ is the required amalgam. Note that JEP amounts to amalgamation over the trivial space, so $(\LLc_\QQb)_1$ has JEP as well. Using Theorem \ref{th:kechros}, we conclude that $\Aut(\LLf_\QQb)$ has a comeager conjugacy class.

In the same way, we can show that $(\LLc_\QQb)_n$ has JEP and WAP for every $n \in \NNb$, i.e., that $\Aut(\LLf_\QQb)$ has ample generics.
\end{proof}

\begin{theorem}
\label{th:weakamalg}
The class $\PPc_\QQb$ does not have the $1$-Hrushovski property, and $(\PPc_\QQb)_1$ does not have WAP. In particular, $\Aut(\GGf_\QQb)$ does not have a comeager conjugacy  class.
\end{theorem}

\begin{proof}
The major part of the proof that $(\PPc_\QQb)_1$ does not have WAP will be accomplished in the next section, where we prove Lemma~\ref{keylemma}. According to this lemma, it is sufficient to find an example of $ \mathcal{O} = (C, f : A \to B) $ for which there is no extension $ D $ of $ C $ that admits an extension $ g : D \to D $ of $ f $ that is a surjective linear isometry. Such an example demonstrates at the same time that $\PPc_\QQb$ does not have the $1$-Hrushovski property. Although the following example is quite simple, perhaps it is not clear where the method comes from. A kind of clarification will be provided later in Proposition~\ref{weakamalgequivalences}.

Let $ C = \mathbb{Q}^{2} $ with the norm $ \Vert (x, y) \Vert = |x| + |y| $. Let $ A = \{ (t, 0) : t \in \mathbb{Q} \} $, $ B = \{ (t, t) : t \in \mathbb{Q} \} $ and $ f(t, 0) = (\frac{1}{2}t, \frac{1}{2}t), t \in \mathbb{Q} $. Assume that suitable $ D $ and $ g $ exist. Since the norm of $ D $ is polyhedral, there is $ n \in \mathbb{N} $ such that $ g^{n} = id_{D} $. Let us consider $ a_{i} = (2^{-i}, 0) $ for $ 0 \leq i \leq n-1 $. Then
\begin{align*}
\Vert a_{0} - f(a_{n-1}) \Vert & = \big\Vert (1, 0) - (\tfrac{1}{2} 2^{-(n-1)}, \tfrac{1}{2} 2^{-(n-1)}) \big\Vert = |1-2^{-n}| + 2^{-n} = 1, \\
\Vert a_{i+1} - f(a_{i}) \Vert & = \big\Vert (2^{-(i+1)}, 0) - (\tfrac{1}{2} 2^{-i}, \tfrac{1}{2} 2^{-i}) \big\Vert \\
 & = |2^{-(i+1)} - 2^{-(i+1)}| + 2^{-(i+1)} = 2^{-(i+1)},
\end{align*}
and so
\begin{align*}
1 & = \Vert a_{0} - f(a_{n-1}) \Vert = \Vert g^{n}(a_{0}) - g(a_{n-1}) \Vert \\
 & = \Big\Vert \sum_{i=0}^{n-2} g^{n-i}(a_{i}) - g^{n-(i+1)}(a_{i+1}) \Big\Vert \\
 & \leq \sum_{i=0}^{n-2} \big\Vert g^{n-i}(a_{i}) - g^{n-(i+1)}(a_{i+1}) \big\Vert = \sum_{i=0}^{n-2} \Vert g(a_{i}) - a_{i+1} \Vert \\
 & = \sum_{i=0}^{n-2} \Vert f(a_{i}) - a_{i+1} \Vert = \sum_{i=0}^{n-2} 2^{-(i+1)} = 1 - 2^{-(n-1)},
\end{align*}
a contradiction.
\end{proof}

In the next proposition, we provide more aspects of the results stated in Theorem~\ref{th:weakamalg}. Here, we can consider both fields $ \mathbb{F} = \mathbb{Q} $ and $ \mathbb{F} = \mathbb{R} $.

\begin{proposition} \label{weakamalgequivalences}
For every $ \mathcal{O} = (C, f : A \to B) \in (\PPc_\FFb)_1 $, the following assertions are equivalent:

(1) There exist $ \mathcal{O}_{1} \in (\PPc_\FFb)_1$ and an embedding $ i : \mathcal{O} \to \mathcal{O}_{1}$ such that for any $ \mathcal{O}_{2} $, $ \mathcal{O}_{3} \in (\PPc_\FFb)_1$ and embeddings $ j : \mathcal{O}_{1} \to \mathcal{O}_{2} $, $ k : \mathcal{O}_{1} \to \mathcal{O}_{3} $, there are $ \mathcal{O}_{4} \in (\PPc_\FFb)_1$ and embeddings $ \widetilde{j} : \mathcal{O}_{2} \to \mathcal{O}_{4} $, $ \widetilde{k} : \mathcal{O}_{3} \to \mathcal{O}_{4} $ such that $ \widetilde{j} \circ j \circ i = \widetilde{k} \circ k \circ i $.

(2) There exists $ \mathcal{O}' \in (\PPc_\FFb)_1 $ of the form $ \mathcal{O}' = (E, g : E \to E) $ that admits an embedding $ i' : \mathcal{O} \to \mathcal{O}' $.

(3) There exists $ n \in \mathbb{N} $ such that for any $ a_{0}, \dots, a_{n-1} \in A $, we have
$$ \Vert a_{0} - f(a_{n-1}) \Vert \leq \sum_{i=0}^{n-2} \Vert a_{i+1} - f(a_{i}) \Vert. $$
\end{proposition}

\begin{proof}
(1) $ \Rightarrow $ (2): This is nothing but Lemma~\ref{keylemma}.

(2) $ \Rightarrow $ (1): We put $ \mathcal{O}_{1} = \mathcal{O}' = (E, g : E \to E) $ and $ i = i' : C \to E $. Let $ \mathcal{O}_{2}=(C_2, f_2 : A_2 \to B_2) $, $ \mathcal{O}_{3}=(C_3, f_3 : A_3 \to B_3) \in (\PPc_\FFb)_1$ and embeddings $ j : \mathcal{O}_{1} \to \mathcal{O}_{2} $, $ k : \mathcal{O}_{1} \to \mathcal{O}_{3} $ be given. Considering the norm $ \Vert (c_{2}, c_{3}) \Vert = \Vert c_{2} \Vert + \Vert c_{3} \Vert $ on $ C_{2} \oplus C_{3} $, we define
$$ C_{4} = (C_{2} \oplus C_{3})/N, \quad \textrm{where $ N = \{ (j(e), -k(e)) : e \in E \} $,} $$
$$ \widetilde{j} : C_{2} \to C_{4}, \quad \widetilde{j}(c_{2}) = (c_{2}, 0) + N, \quad c_{2} \in C_{2}, $$
$$ \widetilde{k} : C_{3} \to C_{4}, \quad \widetilde{k}(c_{3}) = (0, c_{3}) + N, \quad c_{3} \in C_{3}, $$
$$ A_{4} = \mathrm{span} (\widetilde{j}(A_{2}) \cup \widetilde{k}(A_{3})), \quad B_{4} = \mathrm{span} (\widetilde{j}(B_{2}) \cup \widetilde{k}(B_{3})). $$
To show that
$$ f_{4}(\widetilde{j}(a_{2})) = \widetilde{j}(f_{2}(a_{2})), \; a_{2} \in A_{2}, \quad f_{4}(\widetilde{k}(a_{3})) = \widetilde{k}(f_{3}(a_{3})), \; a_{3} \in A_{3}, $$
well-defines a linear mapping $ f_{4} $ on $ A_{4} $, we need to check that
$$ \widetilde{j}(a_{2}) = \widetilde{k}(a_{3}) \quad \Rightarrow \quad \widetilde{j}(f_{2}(a_{2})) = \widetilde{k}(f_{3}(a_{3})), $$
or equivalently $ (a_{2}, -a_{3}) \in N \Rightarrow (f_{2}(a_{2}), -f_{3}(a_{3})) \in N $ for $ a_{2} \in A_{2}, a_{3} \in A_{3} $. For $ (a_{2}, -a_{3}) \in N $, there is $ e \in E $ such that $ a_{2} = j(e) $ and $ a_{3} = k(e) $, and so $ (f_{2}(a_{2}), -f_{3}(a_{3})) = (f_{2}(j(e)), -f_{3}(k(e))) = (j(g(e)), -k(g(e))) \in N $. Thus, $ f_{4} $ is indeed well defined. Let us show that it is an isometry of $ A_{4} $ onto $ B_{4} $. Clearly, $ f_{4}(A_{4}) \subseteq B_{4} $, and regarding the opposite inclusion $ B_{4} \subseteq f_{4}(A_{4}) $, for $ b_{4} \in B_{4} $ of the form $ b_{4} = \widetilde{j}(b_{2}) + \widetilde{k}(b_{3}) $, we can put $ a_{4} = \widetilde{j}(f_{2}^{-1}(b_{2})) + \widetilde{k}(f_{3}^{-1}(b_{3})) $ and obtain $ f_{4}(a_{4}) = b_{4} $. Further, let us pick $ a_{4} \in A_{4} $, so it is of the form
$$ a_{4} = \widetilde{j}(a_{2}) + \widetilde{k}(a_{3}) = (a_{2}, a_{3}) + N, $$
while
$$ f_{4}(a_{4}) = \widetilde{j}(f_{2}(a_{2})) + \widetilde{k}(f_{3}(a_{3})) = (f_{2}(a_{2}), f_{3}(a_{3})) + N. $$
Since $ g $ is a surjective linear isometry on $ E $, it follows that
\begin{align*}
\Vert f_{4}(a_{4}) \Vert & = \inf_{e \in E} \big( \Vert f_{2}(a_{2}) + j(e) \Vert + \Vert f_{3}(a_{3}) - k(e) \Vert \big) \\
 & = \inf_{e \in E} \big( \Vert f_{2}(a_{2}) + j(g(e)) \Vert + \Vert f_{3}(a_{3}) - k(g(e)) \Vert \big) \\
 & = \inf_{e \in E} \big( \Vert f_{2}(a_{2}) + f_{2}(j(e)) \Vert + \Vert f_{3}(a_{3}) - f_{3}(k(e)) \Vert \big) \\
 & = \inf_{e \in E} \big( \Vert a_{2} + j(e) \Vert + \Vert a_{3} - k(e) \Vert \big) \\
 & = \Vert a_{4} \Vert.
\end{align*}

Now, let us check that the choice $ \mathcal{O}_{4} = (C_{4}, f_{4} : A_{4} \to B_{4}) $ works. For $ c_{2} \in C_{2} $, we have
$$ \Vert \widetilde{j}(c_{2}) \Vert = \Vert (c_{2}, 0) + N \Vert = \inf_{e \in E} \big( \Vert c_{2} + j(e) \Vert + \Vert k(e) \Vert \big) = \Vert c_{2} \Vert. $$
Analogously, $ \Vert \widetilde{k}(c_{3}) \Vert = \Vert c_{3} \Vert $ for $ c_{3} \in C_{3} $, and it follows from the definition of $ f_{4} $ that both $ \widetilde{j} $ and $ \widetilde{k} $ are embeddings of $ \mathcal{O}_{2} $, resp.~$ \mathcal{O}_{3} $, into $ \mathcal{O}_{4} $. Finally, for $ e \in E $,
$$ (\widetilde{j} \circ j)(e) = (j(e), 0) + N = (0, k(e)) + N = (\widetilde{k} \circ k)(e), $$
in particular, $ \widetilde{j} \circ j \circ i = \widetilde{k} \circ k \circ i $.

(2) $ \Rightarrow $ (3): We can suppose that $ E $ is an extension of $ C $ and $ g $ is an extension of $ f $. Since the norm of $ E $ is polyhedral, there is $ n \in \mathbb{N} $ such that $ g^{n} = id_{E} $. Given $ a_{0}, \dots, a_{n-1} \in A $, we denote $ x_{i} = g^{n-i}(a_{i}) $, obtaining
\begin{align*}
\Vert a_{0} - f(a_{n-1}) \Vert & = \Vert g^{n}(a_{0}) - g(a_{n-1}) \Vert \\
 & = \Vert x_{0} - x_{n-1} \Vert = \Big\Vert \sum_{i=0}^{n-2} x_{i} - x_{i+1} \Big\Vert \\
 & \leq \sum_{i=0}^{n-2} \Vert x_{i} - x_{i+1} \Vert = \sum_{i=0}^{n-2} \big\Vert g^{n-i}(a_{i}) - g^{n-i-1}(a_{i+1}) \big\Vert \\
 & = \sum_{i=0}^{n-2} \Vert g(a_{i}) - a_{i+1} \Vert = \sum_{i=0}^{n-2} \Vert f(a_{i}) - a_{i+1} \Vert.
\end{align*}

(3) $ \Rightarrow $ (2): Let $ E_{0} $ be the space $ C^{n} $ with the norm $ \Vert (c_{0}, \dots, c_{n-1}) \Vert = \sum_{i=0}^{n-1} \Vert c_{i} \Vert $, and with the isometry
$$ g_{0} : (c_{0}, \dots, c_{n-1}) \mapsto (c_{1}, \dots, c_{n-1}, c_{0}). $$
Let $ N $ be the subspace of $ E_{0} $ generated by vectors of the form
$$ (0, \dots, 0, a, -f(a), 0, \dots, 0), $$
respectively,
$$ (-f(a), 0, \dots, 0, a), $$
where $ a \in A $.

Let $ E $ be the quotient $ E_{0}/N $. Let us denote
$$ \widetilde{c} = (c, 0, \dots, 0) + N, \quad c \in C. $$
Since $ g_{0}(N) = N $, the mapping
$$ g(x+N) = g_{0}(x)+N, \quad x+N \in E, $$
is well-defined, moreover it is an isometry. For $ a \in A $, we have
\begin{align*}
g(\widetilde{a}) & = g \big( (a, 0, \dots, 0) + N \big) = g_{0} \big( (a, 0, \dots, 0) \big) + N \\
 & = (0, \dots, 0, a) + N = (f(a), 0, \dots, 0) + N = \widetilde{f(a)}.
\end{align*}
Hence, it remains to show that $ c \mapsto \widetilde{c} $ is an isometry. For $ c \in C $, we have
$$ \Vert \widetilde{c} \Vert = \inf_{u \in N} \Vert (c, 0, \dots, 0) + u \Vert, $$
thus $ \Vert \widetilde{c} \Vert \leq \Vert c \Vert $ by the choice $ u = 0 $. On the other hand, a general $ u \in N $ is of the form
$$ u = \big( a_{0} - f(a_{n-1}), a_{1} - f(a_{0}), \dots, a_{n-1} - f(a_{n-2}) \big), $$
for which, by the inequality from (3),
\begin{align*}
\Vert (c, 0, \dots, 0) + u \Vert & = \Vert c + a_{0} - f(a_{n-1}) \Vert + \sum_{i=0}^{n-2} \Vert a_{i+1} - f(a_{i}) \Vert \\
 & \geq \Vert c + a_{0} - f(a_{n-1}) \Vert + \Vert a_{0} - f(a_{n-1}) \Vert \geq \Vert c \Vert,
\end{align*}
arriving at $ \Vert \widetilde{c} \Vert \geq \Vert c \Vert $. We obtain $ \Vert \widetilde{c} \Vert = \Vert c \Vert $ for $ c \in C $, finishing the proof of (2).
\end{proof}

To close this section, we show that the two-dimensional example from the proof of Theorem~\ref{th:weakamalg} actually disproves the $1$-Hrushovski property also for the class of general finite-dimensional normed spaces.

\begin{proposition}
Let $ X $ be a finite-dimensional normed linear space containing points $ e_{1}, e_{2} $ such that $ \Vert ae_{1} + be_{2} \Vert = |a| + |b| $ for all $ a, b \in \mathbb{F} $. Then there is no surjective linear isometry on $ X $ that maps $ e_{1} $ to $ \frac{1}{2}(e_{1} + e_{2}) $.
\end{proposition}

\begin{proof}
Suppose that $ T : X \to X $ is a surjective linear isometry satisfying $ Te_{1} = \frac{1}{2}(e_{1} + e_{2}) $. Without loss of generality, we suppose that $ X $ is the linear span of the sequence $ e_{1}, e_{2}, Te_{1}, Te_{2}, T^{2}e_{1}, T^{2}e_{2}, \dots $, since $ T $ maps this span into itself. Let $ n $ be the smallest number such that $ T^{n+1}e_{2} $ belongs to the linear span of $ e_{1}, e_{2}, Te_{2}, T^{2}e_{2}, \dots, T^{n}e_{2} $. We claim that $ e_{1}, e_{2}, Te_{2}, T^{2}e_{2}, \dots, T^{n}e_{2} $ is a basis of $ X $. Since these vectors are linearly independent, we need to show only that each $ T^{m}e_{1} $ and each $ T^{m}e_{2} $ belong to their linear span.

For suitable $ \alpha, \beta_{0}, \beta_{1}, \dots, \beta_{n} $, we have
$$ T^{n+1}e_{2} = \alpha e_{1} + \beta_{0} e_{2} + \beta_{1} Te_{2} + \dots + \beta_{n} T^{n}e_{2}. $$
It is easy to show by induction that
$$ T^{k}e_{1} = \frac{1}{2^{k}} e_{1} + \frac{1}{2^{k}} e_{2} + \frac{1}{2^{k-1}} Te_{2} + \dots + \frac{1}{2} T^{k-1}e_{2}, $$
so
\begin{align*}
T^{n+k+1}e_{2} & = \alpha T^{k}e_{1} + \beta_{0} T^{k}e_{2} + \beta_{1} T^{k+1}e_{2} + \dots + \beta_{n} T^{n+k}e_{2} \\
 & = \alpha \Big[ \frac{1}{2^{k}} e_{1} + \frac{1}{2^{k}} e_{2} + \frac{1}{2^{k-1}} Te_{2} + \dots + \frac{1}{2} T^{k-1}e_{2} \Big] \\
 & \quad + \beta_{0} T^{k}e_{2} + \beta_{1} T^{k+1}e_{2} + \dots + \beta_{n} T^{n+k}e_{2}.
\end{align*}
We first find out that by induction each $ T^{m}e_{2} $ belongs to the linear span of $ e_{1}, e_{2}, Te_{2}, T^{2}e_{2}, \dots, T^{n}e_{2} $, and then it follows that the same holds for each $ T^{m}e_{1} $.

From the formula for $ T^{k}e_{1} $, we see that $ T^{n+1}e_{1} $ is a convex combination of the basic vectors $ e_{1}, e_{2}, Te_{2}, T^{2}e_{2}, \dots, T^{n}e_{2} $, in which each coefficient is positive. Since $ T^{n+1}e_{1} $ has norm $ 1 $, as well as $ e_{1}, e_{2}, Te_{2}, T^{2}e_{2}, \dots, T^{n}e_{2} $, the whole convex hull of $ e_{1}, e_{2}, Te_{2}, T^{2}e_{2}, \dots, T^{n}e_{2} $ consists of vectors of norm $ 1 $. Hence, the norm is smooth in $ T^{n+1}e_{1} $. The same is true for $ e_{1} $, as there is obviously an isometry that maps $ e_{1} $ to $ T^{n+1}e_{1} $. But this is not possible, because the norm is not smooth in $ e_{1} $ due to our assumptions.
\end{proof}

\section{Key lemma}

This section is devoted to the proof of the following lemma important for our proof that $(\PPc_\FFb)_1$ does not have WAP (where $ \mathbb{F} = \mathbb{Q} $ or $ \mathbb{F} = \mathbb{R} $).

\begin{lemma} \label{keylemma}
Let $ \mathcal{O}_{0} = (C_{0}, f_{0} : A_{0} \to B_{0}) \in (\PPc_\FFb)_1$ be such that there exist $ \mathcal{O}_{1} = (C_{1}, f_{1} : A_{1} \to B_{1}) \in (\PPc_\FFb)_1$ and an embedding $ i : \mathcal{O}_{0} \to \mathcal{O}_{1} $ such that for any $ \mathcal{O}_{2}=(C_{2}, f_{2} : A_{2} \to B_{2})$, $ \mathcal{O}_{3}=(C_{3}, f_{3} : A_{3} \to B_{3}) \in (\PPc_\FFb)_1$ and embeddings $ j : \mathcal{O}_{1} \to \mathcal{O}_{2} $, $ k : \mathcal{O}_{1} \to \mathcal{O}_{3} $, there are  $ \mathcal{O}_{4}=(C_{4}, f_{4} : A_{4} \to B_{4}) \in (\PPc_\FFb)_1$ and embeddings $ \widetilde{j} : \mathcal{O}_{2} \to \mathcal{O}_{4} $, $ \widetilde{k} : \mathcal{O}_{3} \to \mathcal{O}_{4} $ such that $ \widetilde{j} \circ j \circ i = \widetilde{k} \circ k \circ i $.

Then there exists $ \mathcal{O}' \in (\PPc_\FFb)_1 $ of the form $ \mathcal{O}' = (E, f : E \to E) $ that admits an embedding $ i' : \mathcal{O}_{0} \to \mathcal{O}' $.
\end{lemma}

\begin{center}
\begin{tikzpicture}
  \node (0) at (-3,0) {$ C_{0} $};
  \node (1) at (-1,0) {$ C_{1} $};
  \node (2) at (0.5,1.5) {$ C_{2} $};
  \node (3) at (0.5,-1.5) {$ C_{3} $};
  \node (4) at (2,0) {$ C_{4} $};
  \draw [->] (0) -- (1);
  \draw [->] (1) -- (2);
  \draw [->] (1) -- (3);
  \draw [->] (2) -- (4);
  \draw [->] (3) -- (4);
  \node (i) at (-2,0.3) {$ i $};
  \node (j) at (-0.5,0.95) {$ j $};
  \node (k) at (-0.5,-0.95) {$ k $};
  \node (jj) at (1.5,1) {$ \widetilde{j} $};
  \node (kk) at (1.5,-0.95) {$ \widetilde{k} $};
\end{tikzpicture}
\end{center}

We start with introducing the following notation:
\begin{itemize}
\item $ \mathfrak{D} $ denotes the space of all finitely supported $ \alpha : \mathbb{Z} \to C_{1} $ with the norm $ \Vert \alpha \Vert_{\mathfrak{D}} = \sum_{k \in \mathbb{Z}} \Vert \alpha(k) \Vert_{C_{1}} $,
\item $ f_{\mathfrak{D}} $ denotes the shift given by $ [f_{\mathfrak{D}}(\alpha)](k) = \alpha(k-1) $,
\item $ \mathfrak{N} \subseteq \mathfrak{D} $ is defined as the subspace generated by all vectors of the form $ (\dots, 0, 0, -f_{1}(a), a, 0, 0, \dots) $, where $ a \in A_{1} $ and it appears at any coordinate,
\item $ \mathfrak{C} = \mathfrak{D}/\mathfrak{N} $,
\item $ f_{\mathfrak{C}} : \mathfrak{C} \to \mathfrak{C} $ is given by $ f_{\mathfrak{C}}(\alpha + \mathfrak{N}) = f_{\mathfrak{D}}(\alpha) + \mathfrak{N} $; this is well defined since $ f_{\mathfrak{D}} $ maps $ \mathfrak{N} $ into $ \mathfrak{N} $, and, moreover, $ f_{\mathfrak{C}} $ is an isometry,
\item $ j_{\mathfrak{C}} : C_{1} \to \mathfrak{C} $ is given by $ j_{\mathfrak{C}}(c) = (\dots, 0, c, 0, \dots) + \mathfrak{N} $, where $ c $ appears on the $ 0 $th coordinate; it follows from the next claim applied to $ k = l = 0 $ that $ j_{\mathfrak{C}} : C_{1} \to \mathfrak{C} $ is an isometry.
\end{itemize}

\begin{claim} \label{cl1}
Let $ k \leq l $ be integers and let $ \mathfrak{D}_{[k, l]} $ and $ \mathfrak{N}_{[k, l]} $ be the subspaces of $ \mathfrak{D} $ and $ \mathfrak{N} $ of all points supported by $ \{ k, k+1, \dots, l \} $. Then
$$ \Vert \alpha + \mathfrak{N} \Vert_{\mathfrak{C}} = \inf \{ \Vert \alpha + \eta \Vert_{\mathfrak{D}} : \eta \in \mathfrak{N}_{[k, l]} \}, \quad \alpha \in \mathfrak{D}_{[k, l]}. $$
Consequently, the subspace $ \{ \alpha + \mathfrak{N} : \alpha \in \mathfrak{D}_{[k, l]} \} $ of $ \mathfrak{C} $ is isometric to $ \mathfrak{D}_{[k, l]}/\mathfrak{N}_{[k, l]} $, and thus it is polyhedral.
\end{claim}

\begin{proof}
Let us fix $ \alpha \in \mathfrak{D}_{[k, l]} $. The inequality ``$ \leq $'' is clear. To prove ``$ \geq $'', given $ \eta \in \mathfrak{N} $, we want to find $ \eta' \in \mathfrak{N}_{[k, l]} $ such that $ \Vert \alpha + \eta' \Vert_{\mathfrak{D}} \leq \Vert \alpha + \eta \Vert_{\mathfrak{D}} $. We can write
$$ \eta = (\dots, 0, - f_{1}(a_{p+1}), a_{p+1} - f_{1}(a_{p+2}), \dots, a_{q-1} - f_{1}(a_{q}), a_{q}, 0, \dots), $$
where $ a_{r} - f_{1}(a_{r+1}) $ appears at the $ r $th coordinate, once $ \eta $ is supported by $ \{ p, p+1, \dots, q \} $. Now, if $ p < k $, then replacing $ a_{p+1} $ by $ 0 $ does not increase $ \Vert \alpha + \eta \Vert_{\mathfrak{D}} $, as
\begin{align*}
\Vert - f_{1}(a_{p+1}) \Vert_{C_{1}} + & \Vert \alpha(p+1) + a_{p+1} - f_{1}(a_{p+2}) \Vert_{C_{1}} \\
 & = \Vert a_{p+1} \Vert_{C_{1}} + \Vert \alpha(p+1) + a_{p+1} - f_{1}(a_{p+2}) \Vert_{C_{1}} \\
 & \geq \Vert \alpha(p+1) - f_{1}(a_{p+2}) \Vert_{C_{1}}.
\end{align*}
Analogously, if $ q > l $, then replacing $ a_{q} $ by $ 0 $ does not increase $ \Vert \alpha + \eta \Vert_{\mathfrak{D}} $, as
\begin{align*}
\Vert \alpha(q-1) + a_{q-1} - & f_{1}(a_{q}) \Vert_{C_{1}} + \Vert a_{q} \Vert_{C_{1}} \\
 & = \Vert \alpha(q-1) + a_{q-1} - f_{1}(a_{q}) \Vert_{C_{1}} + \Vert f_{1}(a_{q}) \Vert_{C_{1}} \\
 & \geq \Vert \alpha(q-1) + a_{q-1} \Vert_{C_{1}}.
\end{align*}
Therefore, we can shrink the support of $ \eta $ until it is a subset of $ \{ k, k+1, \dots, l \} $, without increasing $ \Vert \alpha + \eta \Vert_{\mathfrak{D}} $.
\end{proof}

\begin{claim} \label{cl2}
$ (f_{\mathfrak{C}} \circ j_{\mathfrak{C}})(a) = (j_{\mathfrak{C}} \circ f_{1})(a) $ for each $ a \in A_{1} $.
\end{claim}

\begin{proof}
We have
\begin{align*}
(f_{\mathfrak{C}} \circ j_{\mathfrak{C}})(a) & = f_{\mathfrak{C}}((\dots, 0, a, 0, \dots) + \mathfrak{N}) \\
 & = f_{\mathfrak{D}}((\dots, 0, a, 0, \dots)) + \mathfrak{N} = (\dots, 0, 0, a, \dots) + \mathfrak{N} \\
 & = (\dots, 0, f_{1}(a), 0, \dots) + \mathfrak{N} = j_{\mathfrak{C}}(f_{1}(a)) = (j_{\mathfrak{C}} \circ f_{1})(a)
\end{align*}
for each $ a \in A_{1} $.
\end{proof}

Let us consider a sequence $ \{ u^{*}_{k} \}_{k \in \mathbb{Z}} $ in the dual unit ball $ B_{C_{1}^{*}} $ of $ C_{1} $ such that $ u^{*}_{k}(f_{1}(a)) = u^{*}_{k+1}(a) $ for all $ k \in \mathbb{Z} $ and $ a \in A_{1} $. Then the sequence has the property that $ \sum_{k \in \mathbb{Z}} u^{*}_{k}(\eta(k)) = 0 $ for each $ \eta \in \mathfrak{N} $. It follows that the functional
$$ \mathfrak{u}^{*}(\alpha + \mathfrak{N}) = \sum_{k \in \mathbb{Z}} u^{*}_{k}(\alpha(k)), \quad \alpha \in \mathfrak{D}, $$
is well-defined, moreover, it belongs to $ B_{\mathfrak{C}^{*}} $. Indeed, for $ \alpha \in \mathfrak{D} $, we have $ |\mathfrak{u}^{*}(\alpha + \mathfrak{N})| = |\sum_{k \in \mathbb{Z}} u^{*}_{k}(\alpha(k))| \leq \sum_{k \in \mathbb{Z}} \Vert u^{*}_{k} \Vert_{C_{1}^{*}} \Vert \alpha(k) \Vert_{C_{1}} \leq \sum_{k \in \mathbb{Z}} \Vert \alpha(k) \Vert_{C_{1}} = \Vert \alpha \Vert_{\mathfrak{D}} $, as well as $ |\mathfrak{u}^{*}(\alpha + \mathfrak{N})| \leq \Vert \alpha + \eta \Vert_{\mathfrak{D}} $ for $ \eta \in \mathfrak{N} $, hence $ |\mathfrak{u}^{*}(\alpha + \mathfrak{N})| \leq \inf_{\eta \in \mathfrak{N}} \Vert \alpha + \eta \Vert_{\mathfrak{D}} = \Vert \alpha + \mathfrak{N} \Vert_{\mathfrak{C}} $, which gives $ \Vert \mathfrak{u}^{*} \Vert_{\mathfrak{C}^{*}} \leq 1 $.

Let $ \psi_{1}, \psi_{2}, \dots, \psi_{l} $ be an enumeration of the vertices of $ B_{C_{1}^{*}} $. For $ 1 \leq n \leq l $, we denote $ \psi_{n, 0} = \psi_{n} $ and choose $ \psi_{n, -1}, \psi_{n, 1} \in B_{C_{1}^{*}} $ such that $ \psi_{n, -1}(f_{1}(a)) = \psi_{n, 0}(a) $ and $ \psi_{n, 0}(f_{1}(a)) = \psi_{n, 1}(a) $ for every $ a \in A_{1} $. There are numbers $ s_{n, m} \in \mathbb{F} $ and $ t_{n, m} \in \mathbb{F} $ such that $ \sum_{m=1}^{l} |s_{n, m}| \leq 1 $ and $ \sum_{m=1}^{l} |t_{n, m}| \leq 1 $ for $ 1 \leq n \leq l $ and
$$ \psi_{n, 1} = \sum_{m=1}^{l} s_{n, m} \psi_{m}, \quad \psi_{n, -1} = \sum_{m=1}^{l} t_{n, m} \psi_{m}, \quad 1 \leq n \leq l. $$
Recursively, we define for $ j \geq 1 $
$$ \psi_{n, j+1} = \sum_{m=1}^{l} s_{n, m} \psi_{m, j}, \quad \psi_{n, -j-1} = \sum_{m=1}^{l} t_{n, m} \psi_{m, -j}, \quad 1 \leq n \leq l. $$
By induction, $ \Vert \psi_{n, \pm j} \Vert_{C_{1}^{*}} \leq 1 $ and $ \psi_{n, j}(f_{1}(a)) = \psi_{n, j+1}(a) $, resp.~$ \psi_{n, -j-1}(f_{1}(a)) = \psi_{n, -j}(a) $ for every $ a \in A_{1} $. Therefore, for $ 1 \leq n \leq l $,
$$ \mathfrak{u}^{*}_{n}(\alpha + \mathfrak{N}) = \sum_{k \in \mathbb{Z}} \psi_{n, k}(\alpha(k)), \quad \alpha \in \mathfrak{D}, $$
defines a functional with $ \Vert \mathfrak{u}^{*}_{n} \Vert_{\mathfrak{C}^{*}} \leq 1 $.

\begin{claim} \label{cl3}
For $ c \in i(C_{0}) $, we have
$$ \Vert x \Vert_{\mathfrak{C}} = \sup \big\{ |\mathfrak{u}^{*}_{n}(f_{\mathfrak{C}}^{k}(x))| : 1 \leq n \leq l, k \in \mathbb{Z} \big\}, \quad x \in \mathrm{span} \{ f_{\mathfrak{C}}^{k}(j_{\mathfrak{C}}(c)) : k \in \mathbb{Z} \}. $$
\end{claim}

\begin{proof}
We define
$$ \opnorm{x}_{\mathfrak{C}} = \sup \Big( \big\{ \tfrac{1}{2} \Vert x \Vert_{\mathfrak{C}} \big\} \cup \big\{ |\mathfrak{u}^{*}_{n}(f_{\mathfrak{C}}^{k}(x))| : 1 \leq n \leq l, k \in \mathbb{Z} \big\} \Big), \quad x \in \mathfrak{C}, $$
which is an equivalent norm on $ \mathfrak{C} $ satisfying $ \tfrac{1}{2} \Vert x \Vert_{\mathfrak{C}} \leq \opnorm{x}_{\mathfrak{C}} \leq \Vert x \Vert_{\mathfrak{C}} $, and such that $ f_{\mathfrak{C}} : \mathfrak{C} \to \mathfrak{C} $ is an isometry also with respect to $ \opnorm{\cdot}_{\mathfrak{C}} $. To prove the claim, we show that the norms $ \Vert \cdot \Vert_{\mathfrak{C}} $ and $ \opnorm{\cdot}_{\mathfrak{C}} $ coincide on the linear span of the points $ f_{\mathfrak{C}}^{k}(j_{\mathfrak{C}}(c)), k \in \mathbb{Z} $.

Notice first that $ \Vert j_{\mathfrak{C}}(e) \Vert_{\mathfrak{C}} = \opnorm{j_{\mathfrak{C}}(e)}_{\mathfrak{C}} $ for every $ e \in C_{1} $. Indeed, for some $ n $ with $ 1 \leq n \leq l $, we have $ \psi_{n}(e) = \Vert e \Vert_{C_{1}} $, so $ \Vert j_{\mathfrak{C}}(e) \Vert_{\mathfrak{C}} \geq \opnorm{j_{\mathfrak{C}}(e)}_{\mathfrak{C}} \geq \mathfrak{u}^{*}_{n}(j_{\mathfrak{C}}(e)) = \mathfrak{u}^{*}_{n}((\dots, 0, e, 0, \dots) + \mathfrak{N}) = \psi_{n, 0}(e) = \psi_{n}(e) = \Vert e \Vert_{C_{1}} = \Vert j_{\mathfrak{C}}(e) \Vert_{\mathfrak{C}} $.

It is enough to prove that $ \Vert w \Vert_{\mathfrak{C}} = \opnorm{w}_{\mathfrak{C}} $ for $ w $ of the form
$$ w = \sum_{k=0}^{\nu} w_{k} f_{\mathfrak{C}}^{k}(j_{\mathfrak{C}}(c)). $$
We can assume that $ \nu \geq 1 $. Let us put
$$ C_{2} = C_{3} = \mathrm{span} \bigcup_{k=0}^{\nu} f_{\mathfrak{C}}^{k}(j_{\mathfrak{C}}(C_{1})). $$
We equip $ C_{2} $ with the restriction of the norm $ \Vert \cdot \Vert_{\mathfrak{C}} $ and $ C_{3} $ with the restriction of the norm $ \opnorm{\cdot}_{\mathfrak{C}} $. Then $ C_{2} $ is polyhedral by Claim~\ref{cl1}, and to see that $ C_{3} $ is polyhedral, we prove that
$$ \Vert x \Vert_{C_{3}} = \sup \Big( \big\{ \tfrac{1}{2} \Vert x \Vert_{\mathfrak{C}} \big\} \cup \big\{ |\mathfrak{u}^{*}_{n}(f_{\mathfrak{C}}^{k}(x))| : 1 \leq n \leq l, -\nu \leq k \leq 0 \big\} \Big), \quad x \in C_{3}. $$
Given $ x = \sum_{r=0}^{\nu} f_{\mathfrak{C}}^{r}(j_{\mathfrak{C}}(c_{r})) \in C_{3} $, it is sufficient to show that the function
$$ k \mapsto \max \{ |\mathfrak{u}^{*}_{n}(f_{\mathfrak{C}}^{k}(x))| : 1 \leq n \leq l \} $$
is non-increasing for $ k \geq 0 $ and non-decreasing for $ k \leq -\nu $. Let us note that $ \mathfrak{u}^{*}_{n}(f_{\mathfrak{C}}^{k}(x)) = \mathfrak{u}^{*}_{n}(\sum_{r=0}^{\nu} f_{\mathfrak{C}}^{k+r}(j_{\mathfrak{C}}(c_{r}))) = \sum_{r=0}^{\nu} \psi_{n, k+r}(c_{r}) $ for each $ k \in \mathbb{Z} $. In the case $ k \geq 0 $, we have
\begin{align*}
\mathfrak{u}^{*}_{n}(f_{\mathfrak{C}}^{k+1}(x)) & = \sum_{r=0}^{\nu} \psi_{n, k+1+r}(c_{r}) = \sum_{r=0}^{\nu} \sum_{m=1}^{l} s_{n, m} \psi_{m, k+r}(c_{r}) \\
 & = \sum_{m=1}^{l} s_{n, m} \sum_{r=0}^{\nu} \psi_{m, k+r}(c_{r}) = \sum_{m=1}^{l} s_{n, m} \mathfrak{u}^{*}_{m}(f_{\mathfrak{C}}^{k}(x)),
\end{align*}
and consequently
$$ |\mathfrak{u}^{*}_{n}(f_{\mathfrak{C}}^{k+1}(x))| \leq \sum_{m=1}^{l} |s_{n, m}| |\mathfrak{u}^{*}_{m}(f_{\mathfrak{C}}^{k}(x))| \leq \max \{ |\mathfrak{u}^{*}_{m}(f_{\mathfrak{C}}^{k}(x))| : 1 \leq m \leq l \}. $$
Similarly, for $ k \leq -\nu $, we have
\begin{align*}
\mathfrak{u}^{*}_{n}(f_{\mathfrak{C}}^{k-1}(x)) & = \sum_{r=0}^{\nu} \psi_{n, k-1+r}(c_{r}) = \sum_{r=0}^{\nu} \sum_{m=1}^{l} t_{n, m} \psi_{m, k+r}(c_{r}) \\
 & = \sum_{m=1}^{l} t_{n, m} \sum_{r=0}^{\nu} \psi_{m, k+r}(c_{r}) = \sum_{m=1}^{l} t_{n, m} \mathfrak{u}^{*}_{m}(f_{\mathfrak{C}}^{k}(x)),
\end{align*}
and consequently
$$ |\mathfrak{u}^{*}_{n}(f_{\mathfrak{C}}^{k-1}(x))| \leq \sum_{m=1}^{l} |t_{n, m}| |\mathfrak{u}^{*}_{m}(f_{\mathfrak{C}}^{k}(x))| \leq \max \{ |\mathfrak{u}^{*}_{m}(f_{\mathfrak{C}}^{k}(x))| : 1 \leq m \leq l \}. $$
Hence, the formula for $ \Vert \cdot \Vert_{C_{3}} $ holds, and $ C_{3} $ is polyhedral indeed.

Next, let
$$ A_{2} = A_{3} = \mathrm{span} \bigcup_{k=0}^{\nu-1} f_{\mathfrak{C}}^{k}(j_{\mathfrak{C}}(C_{1})), \quad B_{2} = B_{3} = \mathrm{span} \bigcup_{k=1}^{\nu} f_{\mathfrak{C}}^{k}(j_{\mathfrak{C}}(C_{1})), $$
where we consider the restriction of $ \Vert \cdot \Vert_{\mathfrak{C}} $ on $ A_{2}, B_{2} $ and the restriction of $ \opnorm{\cdot}_{\mathfrak{C}} $ on $ A_{3}, B_{3} $. If we define $ i_{12} = i_{13} = j_{\mathfrak{C}} $ and $ f_{2} = f_{3} = f_{\mathfrak{C}}|_{A_{2}} $, then $ i_{12} : C_{1} \to C_{2}, i_{13} : C_{1} \to C_{3}, f_{2} : A_{2} \to B_{2} $ and $ f_{3} : A_{3} \to B_{3} $ are isometries. By Claim~\ref{cl2}, $ (f_{2} \circ i_{12})(a) = (i_{12} \circ f_{1})(a) $ (equivalently $ (f_{3} \circ i_{13})(a) = (i_{13} \circ f_{1})(a) $) for each $ a \in A_{1} $, and so $ i_{12} $ and $ i_{13} $ are embeddings from $ \mathcal{O}_{1} $ into $ \mathcal{O}_{2} = (C_{2}, f_{2} : A_{2} \to B_{2}) $ and $ \mathcal{O}_{3} = (C_{3}, f_{3} : A_{3} \to B_{3}) $.

\begin{center}
\begin{tikzpicture}
  \node (0) at (-3,0) {$ C_{0} $};
  \node (1) at (-1,0) {$ C_{1} $};
  \node (2) at (0.5,1.5) {$ C_{2} $};
  \node (3) at (0.5,-1.5) {$ C_{3} $};
  \node (4) at (2,0) {$ C_{4} $};
  \draw [->] (0) -- (1);
  \draw [->] (1) -- (2);
  \draw [->] (1) -- (3);
  \draw [->] (2) -- (4);
  \draw [->] (3) -- (4);
  \node (i) at (-2,0.3) {$ i $};
  \node (j) at (-0.55,0.95) {$ i_{12} $};
  \node (k) at (-0.5,-0.95) {$ i_{13} $};
  \node (jj) at (1.5,1) {$ i_{24} $};
  \node (kk) at (1.55,-0.95) {$ i_{34} $};
\end{tikzpicture}
\end{center}

Now, there are $ \mathcal{O}_{4} = (C_{4}, f_{4} : A_{4} \to B_{4}) \in (\PPc_\FFb)_1 $ and embeddings $ i_{24} : \mathcal{O}_{2} \to \mathcal{O}_{4} $, $ i_{34} : \mathcal{O}_{3} \to \mathcal{O}_{4} $ such that $ i_{24} \circ i_{12} \circ i = i_{34} \circ i_{13} \circ i $. Note that $ (i_{24} \circ i_{12})(c) = (i_{34} \circ i_{13})(c) $, that is, $ i_{24}(j_{\mathfrak{C}}(c)) = i_{34}(j_{\mathfrak{C}}(c)) $. By induction, we can prove for $ 0 \leq k \leq \nu $ that
$$ i_{24}(f_{\mathfrak{C}}^{k}(j_{\mathfrak{C}}(c))) = i_{34}(f_{\mathfrak{C}}^{k}(j_{\mathfrak{C}}(c))), $$
since for $ 0 \leq k \leq \nu - 1 $ we can write
\begin{align*}
i_{24}(f_{\mathfrak{C}}^{k+1}(j_{\mathfrak{C}}(c))) & = (i_{24} \circ f_{2})(f_{\mathfrak{C}}^{k}(j_{\mathfrak{C}}(c))) = (f_{4} \circ i_{24})(f_{\mathfrak{C}}^{k}(j_{\mathfrak{C}}(c))) \\
 & = (f_{4} \circ i_{34})(f_{\mathfrak{C}}^{k}(j_{\mathfrak{C}}(c))) = (i_{34} \circ f_{3})(f_{\mathfrak{C}}^{k}(j_{\mathfrak{C}}(c))) \\
 & = i_{34}(f_{\mathfrak{C}}^{k+1}(j_{\mathfrak{C}}(c))).
\end{align*}
Finally, it follows that $ i_{24}(w) = i_{34}(w) $ and $ \Vert w \Vert_{C_{2}} = \Vert i_{24}(w) \Vert_{C_{4}} = \Vert i_{34}(w) \Vert_{C_{4}} = \Vert w \Vert_{C_{3}} $, that is, $ \Vert w \Vert_{\mathfrak{C}} = \opnorm{w}_{\mathfrak{C}} $.
\end{proof}

Before the next claim, let us make one more remark on the above introduced functionals $ \psi_{n, k} $ and numbers $ s_{n, m} $ and $ t_{n, m} $. Let $ S $ be the matrix $ (s_{n, m})_{n,m=1}^{l} $ and $ T $ be the matrix $ (t_{n, m})_{n,m=1}^{l} $, and let $ s^{(j)}_{n, m} $ and $ t^{(j)}_{n, m} $ be the entries of the matrices $ S^{j} $ and $ T^{j} $. Then it is straightforward to show by induction on $ j \geq 0 $ that
$$ \psi_{n, j} = \sum_{m=1}^{l} s^{(j)}_{n, m} \psi_{m}, \quad \psi_{n, -j} = \sum_{m=1}^{l} t^{(j)}_{n, m} \psi_{m}, \quad 1 \leq n \leq l, $$
and
$$ \sum_{m=1}^{l} |s^{(j)}_{n, m}| \leq 1, \quad \sum_{m=1}^{l} |t^{(j)}_{n, m}| \leq 1, \quad 1 \leq n \leq l. $$

\begin{claim} \label{cl4}
For $ c \in i(C_{0}) $, $ \nu \in \mathbb{N} $ and $ \varepsilon > 0 $, there is $ \eta \in \mathfrak{N} $ such that $ \Vert c + \eta(0) \Vert_{C_{1}} + \sum_{0 < |k| \leq \nu} \Vert \eta(k) \Vert_{C_{1}} < \varepsilon $.
\end{claim}

\begin{proof}
We can assume that $ \Vert c \Vert_{C_{1}} = 1 $. Since there is a cluster point of the sequence $ \{ (S^{j}, T^{j}) \}_{j=0}^{\infty} $, we can choose $ \tau > \sigma \geq 0 $ such that $ \tau \geq 3\sigma $, $ \tau \geq \sigma + 2\nu + 1 $ and
$$ |s^{(\tau)}_{n, m} - s^{(\sigma)}_{n, m}| \leq \varepsilon/l, \quad |t^{(\tau)}_{n, m} - t^{(\sigma)}_{n, m}| \leq \varepsilon/l, \quad 1 \leq n, m \leq l. $$
Then, for $ j \geq 0 $, using $ S^{\tau + j} - S^{\sigma + j} = S^{j} (S^{\tau} - S^{\sigma}) $, we obtain $ |s^{(\tau + j)}_{n, m} - s^{(\sigma + j)}_{n, m}| = |\sum_{o=1}^{l} s^{(j)}_{n, o} (s^{(\tau)}_{o, m} - s^{(\sigma)}_{o, m})| \leq (\varepsilon/l) \sum_{o=1}^{l} |s^{(j)}_{n, o}| $. From this and from the same for $ T $, we get
$$ |s^{(\tau + j)}_{n, m} - s^{(\sigma + j)}_{n, m}| \leq \varepsilon/l, \quad |t^{(\tau + j)}_{n, m} - t^{(\sigma + j)}_{n, m}| \leq \varepsilon/l, \quad 1 \leq n, m \leq l. $$

Let us choose $ R \in \mathbb{N} $ with $ R \geq 2/\varepsilon $, and let
$$ x = \sum_{p=0}^{2R-1} (-1)^{p} f_{\mathfrak{C}}^{p(\tau - \sigma)}(j_{\mathfrak{C}}(c)). $$
That is, $ x = \sum_{r=0}^{R-1} [f_{\mathfrak{C}}^{2r(\tau - \sigma)}(j_{\mathfrak{C}}(c)) - f_{\mathfrak{C}}^{(2r+1)(\tau - \sigma)}(j_{\mathfrak{C}}(c))] $, so we have for $ k \in \mathbb{Z} $ and $ 1 \leq n \leq l $ that
$$ \mathfrak{u}^{*}_{n}(f_{\mathfrak{C}}^{k}(x)) = \sum_{r=0}^{R-1} \big[ \psi_{n, 2r(\tau - \sigma) + k}(c) - \psi_{n, (2r+1)(\tau - \sigma) + k}(c) \big]. $$
Let us consider three cases for $ r $:

$ \bullet $ If $ 2r(\tau - \sigma) + k \geq \sigma $, then
\begin{align*}
\psi_{n, 2r(\tau - \sigma) + k}(c) - & \psi_{n, (2r+1)(\tau - \sigma) + k}(c) \\
 & = \sum_{m=1}^{l} s^{(2r(\tau - \sigma) + k)}_{n, m} \psi_{m}(c) - \sum_{m=1}^{l} s^{((2r+1)(\tau - \sigma) + k)}_{n, m} \psi_{m}(c) \\
 & = \sum_{m=1}^{l} (s^{(2r(\tau - \sigma) + k)}_{n, m} - s^{((2r+1)(\tau - \sigma) + k)}_{n, m}) \psi_{m}(c),
\end{align*}
and so
\begin{align*}
\big| \psi_{n, 2r(\tau - \sigma) + k}(c) - \psi_{n, (2r+1)(\tau - \sigma) + k}(c) \big| & \leq \sum_{m=1}^{l} (\varepsilon/l) |\psi_{m}(c)| \\
 & \leq l (\varepsilon/l) \Vert c \Vert_{C_{1}} = \varepsilon.
\end{align*}

$ \bullet $ If $ (2r+1)(\tau - \sigma) + k \leq -\sigma $, then
\begin{align*}
\psi_{n, 2r(\tau - \sigma) + k}(c) - & \psi_{n, (2r+1)(\tau - \sigma) + k}(c) \\
 & = \sum_{m=1}^{l} t^{(-[2r(\tau - \sigma) + k])}_{n, m} \psi_{m}(c) - \sum_{m=1}^{l} t^{(-[(2r+1)(\tau - \sigma) + k])}_{n, m} \psi_{m}(c) \\
 & = \sum_{m=1}^{l} (t^{(-[2r(\tau - \sigma) + k])}_{n, m} - t^{(-[(2r+1)(\tau - \sigma) + k])}_{n, m}) \psi_{m}(c),
\end{align*}
and so
\begin{align*}
\big| \psi_{n, 2r(\tau - \sigma) + k}(c) - \psi_{n, (2r+1)(\tau - \sigma) + k}(c) \big| & \leq \sum_{m=1}^{l} (\varepsilon/l) |\psi_{m}(c)| \\
 & \leq l (\varepsilon/l) \Vert c \Vert_{C_{1}} = \varepsilon.
\end{align*}

$ \bullet $ In the remaining case $ 2r(\tau - \sigma) + k < \sigma $ and $ (2r+1)(\tau - \sigma) + k > -\sigma $, equivalently $ \frac{-\sigma-k}{2(\tau-\sigma)} - \frac{1}{2} < r < \frac{\sigma-k}{2(\tau-\sigma)} $, we use simply
$$ \big| \psi_{n, 2r(\tau - \sigma) + k}(c) - \psi_{n, (2r+1)(\tau - \sigma) + k}(c) \big| \leq 2 \Vert c \Vert_{C_{1}} = 2. $$
This may happen for at most one $ r $, since $ \frac{\sigma-k}{2(\tau-\sigma)} \leq 1 + \frac{-\sigma-k}{2(\tau-\sigma)} - \frac{1}{2} $, as $ \tau \geq 3\sigma $.

Altogether,
\begin{align*}
|\mathfrak{u}^{*}_{n}(f_{\mathfrak{C}}^{k}(x))| & \leq \sum_{r=0}^{R-1} \big| \psi_{n, 2r(\tau - \sigma) + k}(c) - \psi_{n, (2r+1)(\tau - \sigma) + k}(c) \big| \\
 & \leq 2 + (R-1) \varepsilon,
\end{align*}
and using this for all $ 1 \leq n \leq l $ and $ k \in \mathbb{Z} $, we obtain from Claim~\ref{cl3} that
$$ \Vert x \Vert_{\mathfrak{C}} \leq 2 + (R-1) \varepsilon \leq R \varepsilon + (R-1) \varepsilon < 2R \varepsilon. $$
We have $ x = \alpha + \mathfrak{N} $, where $ \alpha \in \mathfrak{D} $ has $ (-1)^{p} c $ on the coordinate $ p(\tau - \sigma) $ for $ 0 \leq p \leq 2R - 1 $, and $ 0 $ elsewhere. For some $ \eta \in \mathfrak{N} $, we have $ \Vert \alpha + \eta \Vert_{\mathfrak{D}} < 2R \varepsilon $. Since $ \tau \geq \sigma + 2\nu + 1 $, we can write
$$ \sum_{p=0}^{2R-1} \Big[ \big\Vert (-1)^{p} c + \eta(p(\tau - \sigma)) \big\Vert_{C_{1}} + \sum_{0 < |k| \leq \nu} \big\Vert \eta(p(\tau - \sigma) + k) \big\Vert_{C_{1}} \Big] $$
$$ \leq \sum_{k \in \mathbb{Z}} \Vert\alpha(k) + \eta(k) \Vert_{C_{1}} = \Vert \alpha + \eta \Vert_{\mathfrak{D}} < 2R \varepsilon, $$
so we must have $ \Vert (-1)^{p} c + \eta(p(\tau - \sigma)) \Vert_{C_{1}} + \sum_{0 < |k| \leq \nu} \Vert \eta(p(\tau - \sigma) + k) \Vert_{C_{1}} < \varepsilon $ for some $ 0 \leq p \leq 2R - 1 $, from which we easily deduce the claim.
\end{proof}

Let us define recursively
$$ E_{0} = F_{0} = C_{1}, \quad E_{\nu+1} = f_{1}^{-1}(B_{1} \cap E_{\nu}), \quad F_{\nu+1} = f_{1}(A_{1} \cap F_{\nu}), $$
and let
$$ E = \bigcap_{\nu=0}^{\infty} E_{\nu}. $$
Then $ E_{\nu} $ is the subspace of $ C_{1} $ of all points on which $ f_{1} $ can be applied $ \nu $ times and $ F_{\nu} $ is the subspace of $ C_{1} $ of all points on which $ f_{1}^{-1} $ can be applied $ \nu $ times, so $ F_{\nu} = f_{1}^{\nu}(E_{\nu}) $.

We have $ E_{0} \supseteq E_{1} \supseteq E_{2} \supseteq \dots $, hence there is $ \nu_{0} $ such that $ E_{\nu} = E $ for each $ \nu \geq \nu_{0} $. It follows that $ f_{1}|_{E} $ is a surjective linear isometry on $ E $. Also, $ F_{\nu} = f_{1}^{\nu}(E_{\nu}) = f_{1}^{\nu}(E) = E $ for each $ \nu \geq \nu_{0} $.

\begin{claim} \label{cl5}
For every $ \nu \in \mathbb{N} \cup \{ 0\} $, there is $ \Gamma_{\nu} > 0 $ such that for each $ \varepsilon > 0 $ and each sequence $ a_{0}, a_{1}, \dots, a_{\nu} $ in $ A_{1} $ with $ \Vert f_{1}(a_{n-1}) - a_{n} \Vert_{C_{1}} < \varepsilon $ for $ 1 \leq n \leq \nu $, the distance of $ a_{0} $ to $ E_{\nu+1} $ is less than $ \Gamma_{\nu} \varepsilon $.

Similarly, for every $ \nu \in \mathbb{N} \cup \{ 0\} $, there is $ \Delta_{\nu} > 0 $ such that for each $ \varepsilon > 0 $ and each sequence $ b_{0}, b_{1}, \dots, b_{\nu} $ in $ B_{1} $ with $ \Vert f_{1}^{-1}(b_{n-1}) - b_{n} \Vert_{C_{1}} < \varepsilon $ for $ 1 \leq n \leq \nu $, the distance of $ b_{0} $ to $ F_{\nu+1} $ is less than $ \Delta_{\nu} \varepsilon $.
\end{claim}

\begin{proof}
We prove only the first part. For $ \nu = 0 $, this is clear, as $ a_{0} \in A_{1} = E_{\nu+1} $. Concerning the induction step, let the statement hold for $ \nu - 1 $ with some $ \Gamma_{\nu-1} $. There is $ \varrho > 0 $ such that for every $ b \in B_{1} $ we have $ \mathrm{dist}(b, B_{1} \cap E_{\nu}) \leq \varrho \, \mathrm{dist}(b, E_{\nu}) $; indeed, let $ \delta $ be the distance of the unit sphere of $ B_{1}/(B_{1} \cap E_{\nu}) $ to $ E_{\nu}/(B_{1} \cap E_{\nu}) $ in $ C_{1}/(B_{1} \cap E_{\nu}) $, then we have
\begin{align*}
\delta \, \mathrm{dist}(b, B_{1} \cap E_{\nu}) & = \delta \Vert b + (B_{1} \cap E_{\nu}) \Vert \\
 &\leq \mathrm{dist}(b + (B_{1} \cap E_{\nu}), E_{\nu}/(B_{1} \cap E_{\nu})) \\
 & = \mathrm{dist}(b, E_{\nu}),
\end{align*}
so we can take $ \varrho = 1/\delta $. By the induction hypothesis, we have $ \mathrm{dist}(a_{1}, E_{\nu}) < \Gamma_{\nu-1} \varepsilon $, and so
\begin{align*}
 \mathrm{dist}(a_{0}, E_{\nu+1}) & = \mathrm{dist}(f_{1}(a_{0}), B_{1} \cap E_{\nu}) \leq \varrho \, \mathrm{dist}(f_{1}(a_{0}), E_{\nu}) \\
 & \leq \varrho \big( \Vert f_{1}(a_{0}) - a_{1} \Vert + \mathrm{dist}(a_{1}, E_{\nu}) \big) < \varrho (1 + \Gamma_{\nu-1}) \varepsilon.
\end{align*}
Hence, we can put $ \Gamma_{\nu} = \varrho (1 + \Gamma_{\nu-1}) $.
\end{proof}

\begin{claim} \label{cl6}
$ i(C_{0}) \subseteq E $.
\end{claim}

\begin{proof}
Let $ c \in i(C_{0}) $, let $ \varepsilon > 0 $ be arbitrary, and let $ \nu $ be large enough to satisfy $ E_{\nu+1} = F_{\nu+1} = E $. By Claim~\ref{cl4}, there is $ \eta \in \mathfrak{N} $ such that $ \Vert c + \eta(0) \Vert_{C_{1}} + \sum_{0 < |k| \leq \nu} \Vert \eta(k) \Vert_{C_{1}} < \varepsilon $. We can write
$$ \eta = \sum_{k \in \mathbb{Z}} (\dots, 0, 0, -f_{1}(a_{k}), a_{k}, 0, 0, \dots), $$
where $ a_{k} \in A_{1} $ appears on the $ k $th coordinate and only finitely many $ a_{k} $'s are non-zero. We have $ \varepsilon > \Vert c + \eta(0) \Vert_{C_{1}} = \Vert c + a_{0} - f_{1}(a_{1}) \Vert_{C_{1}} $ and $ \varepsilon > \Vert \eta(k) \Vert_{C_{1}} = \Vert a_{k} - f_{1}(a_{k+1}) \Vert_{C_{1}} $ for $ 0 < |k| \leq \nu $. Applying Claim~\ref{cl5} on the sequences $ a_{0}, a_{-1}, \dots, a_{-\nu} $ and $ f_{1}(a_{1}), f_{1}(a_{2}), \dots, f_{1}(a_{\nu+1}) $, we obtain $ \mathrm{dist}(a_{0}, E_{\nu+1}) < \Gamma_{\nu} \varepsilon $ and $ \mathrm{dist}(f_{1}(a_{1}), F_{\nu+1}) < \Delta_{\nu} \varepsilon $. It follows that
\begin{align*}
\mathrm{dist}(c, E) & \leq \Vert c + a_{0} - f_{1}(a_{1}) \Vert_{C_{1}} + \mathrm{dist}(a_{0}, E) + \mathrm{dist}(f_{1}(a_{1}), E) \\
 & < (1 + \Gamma_{\nu} + \Delta_{\nu}) \varepsilon.
\end{align*}
Since $ \varepsilon > 0 $ was chosen arbitrarily, we arrive at $ c \in E $.
\end{proof}

To complete the proof of Lemma \ref{keylemma}, we put $ f = f_{1}|_{E} $ and $ i' = i $. Then $ f $ is a surjective linear isometry on $ E $, and Claim~\ref{cl6} guarantees that $ i' $ is an embedding of $ \mathcal{O}_{0} = (C_{0}, f_{0} : A_{0} \to B_{0}) $ into $ (E, f : E \to E) $.

\bigskip

{\bf Acknowledgment.} The authors are grateful to the referee for numerous suggestions that helped to improve this paper.

\par 
\medskip
\textsc{\footnotesize 
	Institute of Mathematics,
	Czech Academy of Sciences,
	\v{Z}itn\'a 25,
	115 67 Praha 1,
	Czech Republic}

\textit{E-mail address}: \texttt{kurka.ondrej@seznam.cz}

\par 
\medskip
\textsc{\footnotesize 
	Institute of Mathematics,
	Polish Academy of Sciences,
	\'Sniadeckich 8,
	00-656 Warszawa,
	Poland}

\textit{E-mail address}: \texttt{mamalicki@gmail.com}

\end{document}